\DeclareSymbolFont{AMSb}{U}{msb}{m}{n}
\documentclass[11pt,a4paper,leqno,]{amsart}
\usepackage[utf8]{inputenc}
\usepackage[english]{babel}
\usepackage{amsmath, amsthm, amstext,amssymb,mathrsfs}
\usepackage{microtype}

\usepackage{indentfirst}
\usepackage[colorlinks,bookmarks]{hyperref}
\usepackage{braket}
\usepackage{caption}
\usepackage{stmaryrd}
\usepackage{comment}
\usepackage{multirow}
\usepackage{booktabs}
\usepackage[all]{xy}
\usepackage{graphicx}

\usepackage{listings}

\newtheorem{cor}{Corollary}[section]
\newtheorem*{cor*}{Corollary}
\newtheorem{lem}[cor]{Lemma}
\newtheorem*{lem*}{Lemma}
\newtheorem{thm}[cor]{Theorem}
\newtheorem*{thm*}{Theorem}

\newtheorem*{conj*}{Conjecture}
\newtheorem{prop}[cor]{Proposition}
\newtheorem*{prop*}{Proposition}

\newcommand{\bt}{\begin{thm}}
\newcommand{\et}{\end{thm}}
\newcommand{\bp}{\begin{prop}}
\newcommand{\ep}{\end{prop}}
\newcommand{\bd}{\begin{defn}}
\newcommand{\ed}{\end{defn}}
\newcommand{\bl}{\begin{lem}}
\newcommand{\el}{\end{lem}}
\newcommand{\bfa}{\begin{fact}}
\newcommand{\efa}{\end{fact}}
\newcommand{\bc}{\begin{cor}}
\newcommand{\ec}{\end{cor}}
\newcommand{\bex}{\begin{example}}
\newcommand{\eex}{\end{example}}
\newcommand{\br}{\begin{remark}}
\newcommand{\er}{\end{remark}}
\newcommand{\ben}{\begin{enumerate}}
\newcommand{\een}{\end{enumerate}}

\theoremstyle{definition}
\newtheorem{defn}[cor]{Definition}
\newtheorem{rmk}[cor]{Remark}
\newtheorem{exa}[cor]{Example}

\newcommand{\bC}{\mathbb{C}}

\newcommand{\bN}{\mathbb{N}}
\newcommand{\bP}{\mathbb{P}}
\newcommand{\bQ}{\mathbb{Q}}

\newcommand{\bZ}{\mathbb{Z}}

\newcommand{\cF}{\mathcal{F}}

\newcommand{\cN}{\mathcal{N}}

\newcommand{\cT}{\mathcal{T}}

\newcommand{\cO}{\mathcal{O}}

\newcommand{\ringS}{H^0_* (\cO_S)}
\newcommand{\ringC}{H^0_* (\cO_C)}

\newcommand{\sF}{\mathscr{F}}

\newcommand{\sI}{\mathscr{I}}

\renewcommand{\dim}{\operatorname{dim}}

\renewcommand{\min}{\operatorname{min}}

\newcommand{\Hom}{\operatorname{Hom}}
\newcommand{\Ann}{\operatorname{Ann}}

\newcommand{\ds}{\displaystyle}

\title{Reconstructing curves from their Hodge classes}

\author{Maria Gioia Cifani, Gian Pietro Pirola and Enrico Schlesinger}
\address[M.G.C.]{Dipartimento di matematica e fisica, Università di Roma 3, Largo San Leonardo Murialdo, 1 – 00146 Roma}
\email{mariagioia.cifani@uniroma3.it}
\address[G.P.P.]{Department of Mathematics 'F. Casorati', University of Pavia, via Ferrata 5, 27100 Pavia, Italy}
\email{gianpietro.pirola@unipv.it}
\address[E.S.]{Dipartimento di Matematica 'F. Brioschi', Politecnico di Milano, Piazza Leonardo da Vinci 32, 20133 Milano, Italy}
\email{enrico.schlesinger@polimi.it}

\begin{document}
\begin{abstract}
Let $S$ be a smooth algebraic surface in $\mathbb{P}^3(\mathbb{C})$. A curve $C$ in $S$ has a cohomology class $\eta_C \in H^1 \hspace{-3pt}\left( \Omega^1_S \right)$. Define $\alpha(C)$ to be the equivalence class of $\eta_C$
in the quotient of $H^1 \hspace{-3pt}\left( \Omega^1_S \right)$ modulo the subspace generated by the class $\eta_H$ of a plane section of
$S$.
In the paper "Reconstructing subvarieties from their periods" the authors Movasati and Sert\"{o}z pose several interesting questions
about the reconstruction of $C$ from the annihilator $I_{\alpha(C)}$ of $\alpha(C)$ in the polynomial ring $R=H^0_*(\cO_{\bP^3})$. It contains the homogeneous ideal of $C$, but is much larger as $R/I_{\alpha(C)}$ is artinian. We give sharp numerical conditions that guarantee $C$ is reconstructed by forms of low degree in $I_{\alpha(C)}$. We also show it is not always the case that the class $\alpha(C)$ is \textit{perfect}, that is, that $I_{\alpha(C)}$ could be bigger than the sum of the Jacobian ideal of $S$ and of the homogeneous ideals of curves $D$ in $S$ for which $I_{\alpha(D)}=I_{\alpha(C)}$.
\end{abstract}
\maketitle

\section{Introduction}
 The Hodge conjecture, one of the most challenging and interesting open questions in algebraic geometry,
can be regarded as a reconstruction problem.
Even when the Hodge conjecture is known, as for curves on surfaces,
there are a series of somewhat related problems that might shed a new light on some aspects of the cycle map.

A good example is given by  \cite[Theorem 4.b.26]{GriffithsHarris} where
it is proven that, given a smooth surface $S \subset \bP^3$ and an integral class $\gamma$ in $H^1 \hspace{-3pt}\left(  \Omega^1_S \right)$ with the same numerical properties as the fundamental class of a curve $C \subset S$, then $\gamma$ is itself the fundamental class of an effective divisor $D \subset S$ provided $\deg(S)$ is large relative to
the self-intersection of $\gamma$ and to $\deg(C)$.

In a similar vein, very interesting recent work by Movasati and Sert\"{o}z \cite{MovasatiSertoz} concerns the reconstruction of subvarieties of $\bP^N$ from their periods.
Our purpose is to give an answer, in the special case of curves lying on a smooth algebraic surface $S$ in complex projective space, to two questions raised in \cite{MovasatiSertoz} that we now illustrate.
A curve $C$ in $S$ has a fundamental cohomology class $\eta_C \in H^1\hspace{-3pt}\left( \Omega^1_S \right)$. We denote by $\alpha(C)$ the equivalence class of $\eta_C$
in the quotient of $H^1\hspace{-3pt}\left( \Omega^1_S \right)$ modulo the subspace generated by the class $\eta_H$ of a plane section of
$S$: the class $\alpha(C)$ depends on the embedding of $S$ in $\mathbb{P}^3$, and can be seen as a linear form on the primitive cohomology
$H^1\hspace{-3pt}\left( \Omega^1_S \right)^{\perp_H}$.
Following \cite{MovasatiSertoz} we focus our analysis on the {\em annihilator} $I_{\alpha(C)}$ of $\alpha(C)$ in the polynomial ring $R=H^0_*(\cO_{\bP^3})$. Note - see Proposition \ref{vanishing} - that $I_{\alpha(C)}=I_{\alpha(D)}$ for two curves $C$ and $D$ in $S$ if and only if
$mC+nD+pH$ is linearly equivalent to zero for some choice of integers $m$, $n$ and $p$
with $m$ and $n$ non zero and relatively prime.
Thus the annihilator $I_{\alpha(C)}$, which contains the homogeneous ideal $I_C$ of $C$ and
the Jacobian ideal $J_S$ of $S$, in general it is much larger than $I_C+J_S$,
as it contains the ideal $I_D$ for any curve $D$ for which there is a relation
$mC+nD+pH \sim 0$ as above. Still, one can ask whether $C$ can be reconstructed from
$I_{\alpha(C)}$ when $\deg(S)$ is large with respect to the degree or other invariants of $C$, and Movasati and Sert\"{o}z in \cite{MovasatiSertoz} investigate, in a more general context than ours, the following questions:
\begin{enumerate}
  \item under which conditions $I_{\alpha(C)}$ reconstructs $C$, in the sense that
forms of low degree in  $I_{\alpha(C)}$ cut out the curve $C$ scheme-theoretically?
To be precise, we will say that $C$ is {\em reconstructed at level $m$ by $I_{\alpha(C)}$} if its homogeneous ideal $I_C$ is generated over $R$ by $I_{\alpha(C),\leq m}$ - that is, by forms of degree $\leq m$ in
$I_{\alpha(C)}$.

  \item as the example of complete intersection suggests, they define a class $\alpha \in H^1\hspace{-3pt}\left( \Omega^1_S \right)/{\mathbb{C} \,\eta_H}$
  to be \emph{perfect at level $m$} if there exist effective divisors $D_1,\ldots,D_q$
in $S$ such that $I_{\alpha(D_i)}=I_{\alpha}$ for every
$i=1,\ldots q$ and
$$
I_{\alpha,j} = \sum_{i=1}^q I_{{D_i},j}+J_{S,j} \quad \mbox{for every $j \leq m$}
$$
where $J_{S}$ denotes the Jacobian ideal of $S$. The question is under which conditions  the class $\alpha(C)$ is perfect at level $m$, and whether all classes
$\alpha(C)$ are perfect at every level $m$.
\end{enumerate}

In this paper we prove two theorems that give partial answers to these questions.
Our first theorem extends known results on complete intersections \cite{Dan,MovasatiSertoz}
to arithmetically Cohen-Macaulay curves (ACM curves for short). The  tools we need for this are provided by a very nice paper by Ellingsrud and Peskine \cite{EP} which unfortunately seems to be little known.
In \cite{EP} the authors were interested in the study of the Noether-Lefschetz locus, and the invariant
$\alpha(C)$ plays a prominent role in their work because it vanishes if and only if the curve is a complete intersection
of $S$ and another surface. Their paper connects the class $\alpha(C)$ to the normal sequence arising from the inclusions $C \subset S \subset \bP^3$ and gives an effective tool for computing its annihilator $I_{\alpha(C)}$ - see Lemma \ref{propannihilators}. To state our first theorem,
given a curve $C$ in $\bP^3$, we let $s(C)$ be the minimum degree of a surface containing $C$, and $e(C)$ the index of speciality of $C$, that is, the maximum $n$ such that
$\cO_C(n)$ is special, that is, $h^1(\cO_C(n)))>0$.

\bt \label{maintheorem}
Suppose $C$ is an ACM curve on the smooth surface $S \subseteq \bP^3 (\bC)$. Let $s$ denote the degree of $S$.
Then
\begin{enumerate}
  \item if  $s\geq 2e(C)+8-s(C)$, the curve $C$  is reconstructed  at level $e(C)+3$ by
 $I_{\alpha(C)}$;
  \item the class $\alpha(C)$ is perfect at level $m$ for every $m$.
\end{enumerate}
\et

For example, let $C$ be a twisted cubic curve: $C$ is then ACM with invariants $s(C)=2$ and $e(C)=-1$. By Theorem \ref{maintheorem}, if $S$ is a quartic surface containing $C$, then $C$ is cut out scheme-theoretically by the quadrics
whose equations lie in $I_{\alpha(C,S)}$. This was suggested  and verified for thousands of randomly chosen quartic surfaces
containing $C$ in \cite[Sections 2.3 and 3.2]{MovasatiSertoz}. 

Our second theorem provides a first example of a non-perfect algebraic class $\alpha(C)$, giving a negative answer to Question 2.12 in \cite{MovasatiSertoz}.

\bt \label{intro-rationalquintic}
Let $C \subset \bP^3$ be a smooth rational curve of degree $4$ contained in a
smooth surface $S$ of degree $s=4$. The class $\alpha(C)$ in $S$ is not
perfect at level $3$.
\et

It would be very interesting to determine conditions for a class $\alpha(C)$ to be perfect,
and we don't know whether ACM curves form the largest set of curves $C$ whose
classes $\alpha(C)$, in any smooth surface $S$ containing $C$, are perfect.

Finally, we note that Movasati and Sert\"{o}z pose their questions of reconstruction and perfectness in a more general context, namely for classes in $H^n(\Omega_X^n)$ of varieties of dimension $n$ in smooth hypersurfaces $X$ in $\bP^{2n+1}$. An interesting and challenging problem is trying to answer those question for every $n$, generalizing as far as it is possible the results of this paper to higher dimension.
\medskip

The paper is structured as follows. In section \ref{2} we collect some well known facts we need,
and, for the benefit of the reader, we recall in some detail the constructions from \cite{EP} we will need in the sequel of the paper. In section \ref{3} we prove Proposition \ref{condizioni} - a numerical criterion that guarantees, when the degree of $S$ is large to respect to that of $C$, that the curve $C$ is reconstructed  at a certain level $m$ by
 $I_{\alpha(C)}$. As an example we prove in Corollary \ref{rationalcurves} a reconstruction result for a general rational curve of degree $d$.  In section \ref{4} we prove
Theorem \ref{maintheorem}, which is split in Theorems \ref{acmcurves} and \ref{acmperfette}. In section \ref{sec:rational} we prove Theorem \ref{intro-rationalquintic}.

\thanks{The authors are members of GNSAGA of INdAM.
The authors were partially supported by national MIUR funds,
PRIN  2017 Moduli and Lie theory,  and by MIUR: Dipartimenti di Eccellenza Program
   (2018-2022) - Dept. of Math. Univ. of Pavia.
    }
\section{Preliminaries}\label{2}
We work in the projective space $\bP^3$ over the field $\bC$ of complex numbers.
Given a coherent sheaf $\sF$ on $\bP^3$ and $i \in \bN$, we define
$$
H^i_* (\cF)= \bigoplus_{n \in \bN} H^i \hspace{-3pt} \left(\bP^3, \, \cF(n) \right).
$$
These are graded module over the polynomial ring
$$
R=H^0_* (\cO_{\bP^3}) \cong \bC[x,y,z,w].
$$
Given a subscheme $X$ of $\bP^3$, we will denote by $\sI_X$ its sheaf of ideals, and by $I_X = H^0_* (\sI_X)$
its saturated homogeneous ideal in $R$. We will write $I_{X,n}$ to denote its $n^{th}$ graded piece
$H^0(\bP^3,\sI_X(n))$.

If  $M=\bigoplus_{n\in \mathbb{Z}} M_n$ a graded $R$-module,
the {\em graded $\bC$-dual} module $M^*$ of $M$ is defined by setting $(M^*)_m =\Hom_{\bC}(M_{-m}, \bC)$
with multiplication $R_n \times  (M^*)_m \to (M^*)_{m+n}$ defined by
$$g\lambda (v) = \lambda(gv),\ \ \forall g \in R_n, \lambda \in (M^*)_m, v \in M_{-m-n}.$$

\noindent
By Serre's duality, if $X \subseteq \bP^N$ is an equidimensional Cohen-Macaulay subscheme of dimension $d$, then for any locally free sheaf
$\cF$ on $X$ there is an isomorphism of graded $R$-modules
$$
\left(H^i_* (X,\cF)\right)^* \cong H^{d-i}_* (X,\cF^\vee \otimes \omega_X)
$$

Let $S$ be a {\em smooth algebraic surface} of degree $s$ in $\bP^3$,
and $C \subset S$ a {\em curve}, that is,
an effective Cartier divisor in $S$.
The curve $C$ has a cohomology class $\eta_C \in H^1\hspace{-3pt}\left(S, \Omega^1_S \right)$.
It can be defined as follows: the curve $C$ defines a linear form $\lambda_C$ on the set of $(1,1)$ forms by integration; abstractly
one can define this linear form as the image of the trace map $ H^1\hspace{-3pt}\left(C, \Omega^1_C \right) \rightarrow \bC$ under the transpose
of the morphism  $H^1\hspace{-3pt}\left(S, \Omega^1_S \right) \rightarrow H^1\hspace{-3pt}\left(C, \Omega^1_C \right)$ obtained by restricting differentials
on $S$ to $C$ \cite[Chapter III Ex. 7.4]{Hartshorne}. The cohomology class $\eta_C$ is the image of $\lambda_C$
under the Serre's duality isomorphism $ H^1\hspace{-3pt}\left( \Omega^1_S \right)^* \cong H^1\hspace{-3pt}\left( \Omega^1_S \right)$.

If $\cO_S(C)$ denotes the invertible sheaf on $S$ corresponding to $C$, then $\eta_C=c(\cO_S(C))$ where $c$ denotes the first Chern class homomorphism
\begin{equation}\label{classMap}
c: \mbox{Pic} (S) \to H^1\hspace{-3pt}\left( S,\Omega^1_S \right).
\end{equation}

The perfect pairing $\langle \quad, \quad\rangle$ of Serre's duality
is compatible with the intersection product of divisor classes \cite[Chapter V Ex. 1.8]{Hartshorne} in the sense that for every pair of Cartier divisors $D$ and $E$ on $S$
$$
\langle c(\cO_S(D)), c(\cO_S(E)) \rangle = D\cdot E.
$$
Since $S$ is a surface in $\bP^3$, numerically equivalent divisors on $S$ are
linearly equivalent, and the first Chern class map $\mbox{Pic} (S) \to H^1\hspace{-3pt}\left( S,\Omega^1_S \right)$ is injective.

The cotangent bundles of $S$ and $\bP^3$ are related by the exact sequence
\begin{equation}\label{Cotangent}
     0 \to \cO_S(-s) \to \Omega^1_{\bP^3} \otimes \cO_S \to \Omega^1_S \to 0.
\end{equation}
It is well known (see e.g. \cite{EGPS}) that $H^1(\Omega^1_{\bP^3} \otimes \cO_S) \cong \bC$ and that its image
in $ H^1(\Omega^1_S)$ is the class $\eta_H$ of a plane section $H$ of $S$. We look at a portion of the long cohomology sequence arising from (\ref{Cotangent})
\begin{equation}\label{delta}
 H^1\hspace{-3pt}\left( \Omega^1_S \right) \stackrel{\delta}{\to} H^2\hspace{-3pt}\left( \cO_S(-s) \right)
 \stackrel{\epsilon}{\to}  H^2 \hspace{-3pt}\left(\Omega^1_{\bP^3} \otimes \cO_S \right)
\end{equation}
Dualizing and using Serre's duality we get an exact sequence
\begin{equation}\label{epsilon}
 H^0\hspace{-3pt}\left( \mathcal{T}_{\bP^3}(s-4) \right) \stackrel{\epsilon^*}{\to}
 H^0\hspace{-3pt}\left( \cO_S(2s-4) \right)
 \to \mbox{Im} (\delta)^*\to 0
\end{equation}
We denote by $J_S$ the Jacobian ideal of $S$, that is, the ideal of $R$ generated by the partial derivatives of an equation of $S$. Then the above discussion is summarized by Griffith's
theorem:  the primitive first cohomology group of $S$ is isomorphic to the $(2s-4)$-graded
piece of the Jacobian ring of $S$:
$$
H^1\hspace{-3pt}\left( \Omega^1_S \right)^{\perp_H} \cong \mbox{Im}(\delta)^* \cong \dfrac{H^0\hspace{-3pt}\left(\cO_{\bP^3}(2s-4)\right)}{J_{S,2s-4}}
$$



\begin{defn} \label{maindef}
Given a curve $C$ in $S$, we will denote by $\alpha(C)=\alpha(C,S,\bP^3)$ the image of its cohomology class $\eta_C$ under the map
$$
 H^1\hspace{-3pt}\left( \Omega^1_S \right) \stackrel{\delta}{\to} H^2\hspace{-3pt}\left( \cO_S(-s) \right)
\cong  H^0\hspace{-3pt}\left( \cO_S(2s-4) \right)^*
$$
Thus $\alpha(C)$ is a linear form on $H^0\hspace{-3pt}\left( \cO_S(2s-4) \right)$ that vanishes
on $J_{S,2s-4}$.

Given $\alpha \in H^0 (\cO_S(2s-4))^*$, we denote by $I_{\alpha}$ the annihilator
of $\alpha$ in the {\em polynomial ring $R$}: it is the homogeneous ideal in $R$ whose $n^{th}$ graded piece is
$$I_{\alpha,n}=\{f \in R_n\ |\ \alpha(fg)=0,\ \forall g \in H^{0} (\cO_S (2s-4-n))\ \}.$$
\end{defn}

\begin{rmk}
When writing the paper, we decided to take all ideals in the polynomial ring $R= H^0_* (\cO_{\bP^3})$: thus $J_S$ and $I_\alpha$ are for us ideals of $R$, and $J_S \subset I_{\alpha(C)}$. Our motivation is that we would like to compare $I_\alpha$ with the ideal of $C$ as a curve in $\bP^3$.
In \cite{MovasatiSertoz} the author's denote by $I_\alpha$ the annihilator of $\alpha$ in the Jacobian ring and by $\tilde{I_{\alpha}}$ its preimage in $R$.
\end{rmk}

Let $T=R/I_S=H^0_* (\cO_S)$. Then $\alpha \in (T_{2s-4})^*$, and
 the ideal $I_{\alpha}$ is determined by $\mbox{Ker} (\alpha) \subseteq T_{2s-4}$; conversely, one can recover
$\mbox{Ker} (\alpha)$ as the image of  $I_{\alpha,2s-4}$ via the quotient map $R_{2s-4} \rightarrow T_{2s-4}$.
The perfect pairing
$$R_n /I_{\alpha ,n}  \times \left(R_{2s-4-n} / I_{\alpha,2s-4-n} \right)^* \to \bC$$
shows
$A:= R/I_{\alpha} = \bigoplus_{n=0}^{2s-4} A_n$ is an artinian Gorenstein ring of socle $2s-4$ \cite[Prop 1.3]{EP}.


In \cite{EP} the authors were interested in the study of the Noether-Lefschetz locus, and the invariant
$\alpha(C)$ plays a prominent role in their work because it vanishes if and only if the curve is a complete intersection
of $S$ and another surface. More generally, a Lefschetz type theorem about the Picard group of $S$
(see \cite{SGA7II,Badescu, Voisin}, ) implies the following
fact:

\begin{prop} \label{vanishing}
Let $C$ and $D$ be effective divisors on a smooth surface $S \in\bP^3$, and let $H$ denote a plane section of $S$.
Then $I_{\alpha(C)} = I_{\alpha(D)}$ if and only if there exist $m,n,p \in \bZ$, $m,n \neq 0$ and relatively prime, such that
$mC+nD+pH$ is linearly equivalent to zero.
\end{prop}

\begin{proof}
Suppose $mC+nD+pH$ is linearly equivalent to zero and $m$ and $n$ are nonzero.
The cotangent complex (\ref{Cotangent}) gives rise to an exact sequence in
cohomology
\begin{equation}\label{cohomseq}
H^1\hspace{-3pt}\left( \Omega^1_{\bP^3} \otimes \cO_S \right) \cong \bC \stackrel{\gamma}{\to}
H^1\hspace{-3pt}\left( \Omega^1_S \right) \stackrel{\delta}{\to} H^2\hspace{-3pt}\left( \cO_S(-s) \right) \simeq H^0\hspace{-3pt}\left( \cO_S(2s-4)\right)^*
\end{equation}
and one knows that $\gamma(1)=\eta_H$, so that the kernel of $\delta$ is the $\bC$-line spanned by $\eta_H$. From
$mC+nD+pH \sim 0$ we then deduce $m \alpha (C)=-n \alpha(D)$. Since $m$ and $n$ are nonzero, the linear forms  $\alpha (C)$ and
$\alpha(D)$ have the same kernel, hence $I_{\alpha(C)}  = I_{\alpha(D)} $.

In the other direction, suppose $I_{\alpha(C)}  = I_{\alpha(D)} $, that is,  $\alpha (C)$ and
$\alpha(D)$ have the same kernel. Then $\alpha (C)= c \, \alpha (D)$ for a nonzero {\em complex} number $c$. Using (\ref{cohomseq}) and the intersection pairing we deduce that
are integers $m$, $n$, $p$, with $m$ and $n$ nonzero,such that
$mC+nD+pH$ is linearly equivalent to zero. Finally, $m$ and $n$ can be taken relatively prime
because $Pic(S)/\bZ H$ has no torsion (see for example \cite[Theorem B]{Badescu}). In particular, when $D=0$, one can take $m=1$.
\end{proof}

As noted in \cite{EP} and \cite[Lemma 2.3]{MovasatiSertoz}, the ideal $I_{\alpha(C)}$ contains
the ideal of $C$ in $S$. This follows from the remark of \cite{EP} that $\alpha (C) \in H^0\hspace{-3pt}\left( \cO_S(2s-4)\right)^*$ is the pull-back of a linear form
$\beta(C) \in H^0\hspace{-3pt}\left( \cO_C(2s-4)\right)^*$. For the benefit of the reader and for later use, we give a proof of this fact. The linear form $\beta(C)$ arises from the normal bundles exact sequence:
\begin{equation}\label{Normal}
 0 \to  \cN_{C/S} \cong \omega_C(4-s) \to \cN_{C/\bP^3} \to  \cN_{S/\bP^3} \otimes \cO_C \cong  \cO_{C}(s) \to 0.
\end{equation}
Tensoring (\ref{Normal}) with $\cO_C(-s)$ and taking cohomology we obtain a map $H^0(\mathcal{O}_C) \to H^1(\omega_C(4-2s))$
and we let $$\beta(C) \in H^0\hspace{-3pt}\left( \cO_C(2s-4)\right)^* \cong H^1(\omega_C(4-2s))$$ denote the image
of $1 \in H^0(\mathcal{O}_C)$.

\begin{prop}\cite[Construction 1.8]{EP} \label{pullback}
The linear form $\alpha(C)$ is the pull-back of $\beta(C)$ to $S$, that is,
$\alpha(C)= \rho^*(\beta(C))$ where $\rho^*$ is the transpose of the natural map
$\rho: H^0\hspace{-3pt}\left( \cO_S(2s-4)\right) \rightarrow H^0\hspace{-3pt}\left( \cO_C(2s-4)\right)$.
\end{prop}

\begin{proof}
Observe
that $\Omega^1_S$ is a rank two vector bundle with determinant $\omega_S$, hence the tangent bundle
$\cT_S=(\Omega^1_S)^\vee $
is isomorphic to $\Omega^1_S \otimes \omega_S^{-1}= \Omega^1_S(4-s)$. The tangent complex of $S \subseteq \bP^3$ and the normal
bundle sequence (\ref{Normal}) give rise to a commutative diagram

\begin{equation*}
\xymatrix{
  0 \ar[r] & \Omega^1_S \cong \cT_S (s-4) \ar[d] \ar[r] & \cT_{\bP^3} \otimes \cO_S(s-4) \ar[d] \ar[r]
  & \cO_S(2s-4) \ar[d] \ar[r] &0
    \\
  0 \ar[r] &\omega_C  \ar[r] & \cN_{C/\bP^3} (s-4) \ar[r]
  & \cO_C(2s-4) \ar[r] &0
 }
\end{equation*}
Taking cohomology and dualizing one sees that $\alpha(C)$ is the pull back of $\beta(C)$ to $S$.
\end{proof}
The following Lemma in \cite{EP} gives an effective method  to compute $I_{\alpha}$ in many cases.
\begin{lem}{ \cite[Lemma 1.10]{EP}} \label{propannihilators}
Let $N(C)$ denote the image of the map
$
H^0_* \cN_{C/\bP^3} (-s) \to \ringC
$
arising from the normal bundle sequence (\ref{Normal}).
Let $\pi:R=H^0_* (\cO_{\bP^3}) \to \ringC$ be the natural map.
Then, for every integer $n$,
$$\pi^{-1}\left(N(C)_n\right) \subseteq I_{\alpha(C),n}$$
with  equality if $\pi_{2s-4-n}$ is surjective.
\end{lem}

\begin{proof}
The exact sequence
\begin{equation*}
\xymatrix{
H^0_*  \cN_{C/\bP^3} (-s) \ar[r] & \ringC \ar[r]^<{\qquad 1 \mapsto \beta} &
\left(H^0_*(\cO_C(2s-4)) \right)^*
 }
\end{equation*}
shows $N(C)=\Ann_{\ringC}(\beta)$.

The map $\pi:R \to \ringC$ factors through $\rho: \ringS \to \ringC$.
To simplify notation, write $T=\ringS$ and $e=2s-4$. As $\alpha$ is an element of the $T$-module
$T^*$, the ideal $I_\alpha$, which by definition is the annihilator of $\alpha$ in $R$,
is the inverse image of $\Ann_{T}(\alpha)$ under the surjective map $R \to T$. Hence
what we have to prove is that
$\rho^{-1}\left(N(C)_n\right) \subseteq \Ann_{T}(\alpha)_n$  for every integer $n$,
with equality holding when $\rho_{e-n}$ is surjective. Now
$$
\Ann_T (\alpha=\rho^*(\beta))_n= \left\{
g \in T_n: g\rho^*(\beta)(v)=\beta (\rho(g) \rho(v))=0 \quad \forall v \in T_{e-n}
\right\}
$$
while the inverse image $\rho^{-1}\left(N(C)_n\right) $ of the $n^{th}$
graded piece of the annihilator of $\beta(C)$ in $\ringC$ is equal to
$$
\left\{
g \in T_n: (\rho(g) \beta)(w)=\beta (\rho(g) w)=0 \quad \forall w \in H^0(\cO_C(e-n))
\right\}.
$$
The thesis is now evident.
\end{proof}

\bc
The annihilator $I_{\alpha(C)}$ of $\alpha(C)$
contains both the homogeneous ideal of $C$ and the Jacobian ideal of the surface $S$.
\ec

To exemplify the scope of this construction, we remark that it immediately yields the following well known corollary (originally due to Griffiths and Harris, see \cite{EGPS}
for more details).
\bc
Suppose $S$ is a smooth surface in $\bP^3$ and $C$ is an effective divisor on $S$. Then
$C$ is a complete intersection of $S$ and another surface if and only if the sequence \ref{Normal} of normal bundles splits.
\ec
\begin{proof}
If $C$ is a complete intersection of $S$ and another surface, it is clear that the sequence splits. Conversely,if the sequence splits, then $\beta(C)=0$. Therefore $\alpha(C)=0$, and the thesis follows from Proposition \ref{vanishing}.
\end{proof}



\section{Reconstruction of the ideal}\label{3}

Motivated by \cite{MovasatiSertoz}, we want to compare $I_C$ and $I_{\alpha(C)}$.
The following proposition gives rather sharp sufficient conditions for the curve $C$ to be reconstructed  at level $p$ by
 $I_{\alpha(C)}$.

\begin{prop}\label{condizioni}
Let $S$ be a smooth surface of degree $s$ in $\bP^3$, and let $C$ be an effective Cartier divisor on $S$.
Assume that the homogeneous ideal $I_C$ is generated by its forms of degree $\leq p$ and that the following vanishing conditions are satisfied
\begin{enumerate}
    \item $h^1( \sI_C(2s-4-p))=0$
    \item $h^0(\cN_{C/\bP^3}(p-s))=0$
\end{enumerate}
then $I_{\alpha(C),p}= I_{C,p}$, therefore $C$ is reconstructed  at level $p$ by
 $I_{\alpha(C)}$.
\end{prop}
\begin{proof}
Since $h^0\hspace{-3pt}\left(\cN_{C/\bP^3}(p-s)\right)=0$, the annihilator of $\beta(C)$ in degree $p$ vanishes.
Since $\pi_{2s-4-p}: R_{2s-4-p} \to H^0 \cO_C(2s-4-p))$  is surjective, by Lemma \ref{propannihilators}
$$I_{\alpha(C),p}= \pi_p^{-1} (\Ann (\beta_C)_p)= I_{C,p}.$$
\end{proof}

We can now answer a question raised in \cite[Section
2.3.1]{MovasatiSertoz} about twisted cubics contained in quartic surfaces:
if $C$ is a twisted cubic contained
in a smooth quartic
surface $S \subset \bP^3$,  then $C$ is cut out by quadrics in $I_{\alpha(C)}$.
More generally:

\begin{cor} \label{rationalcurves}
Suppose $C \subset \bP^3$ is a general rational curve of degree $d \geq 3$
and let $n_0$ be the round up of $\sqrt{6d-2}-3$, that is, the smallest positive integer $n$ such
that $\binom{n+3}{3}-nd -1 \geq 0$. If $C$ is contained in a smooth surface $S$ of degree $ s \geq n_0+3$, then $C$ is reconstructed  at level $n_0\!+\!1$ by $I_{\alpha(C,S)}$.
\end{cor}

\begin{proof}
By \cite{harthirshratl} a general rational curve is a curve of maximal rank, that is,
$h^0 (\sI_{C}(n))=0$ for $n \leq n_0-1$ and $h^1 (\sI_{C}(n))=0$ for $n \geq n_0$.
Hence $C$ is $n_0+1$ regular in the sense of Castelnuovo-Mumford, and $I_C$ is generated
by its forms of degree $\leq n_0+1$. Furthermore, by \cite{Eisenbud-VdVen1}
the normal bundle of the immersion
$\bP^1 \to C \subset \bP^3$ is isomorphic to $\cO_{\bP^1}(2d\!-\!1) \oplus \cO_{\bP^1}(2d\!-\!1)$. Hence  $h^0(\cN_{C/\bP^3}(-m))=0$ for every $m \leq -2$.
 Thus we can apply Proposition \ref{condizioni} with
$p=n_0\!+\!1$.
\end{proof}

%

\begin{rmk}
If $C$ is a smooth irreducible curve of degree $d$, then
$h^1\left( \sI_C(n)\right)=0$ for every $n \geq d\!-\!3-e$ (see \cite{GPspeciality} and \cite{Han}),
where $e:=e(C)=\max \{n\ |\ h^1\left(\cO (n)\right) >0\}$ is the \textit{index of speciality} of $C$.
\end{rmk}

\begin{cor}
\label{boundons}
Let $S$ be a smooth surface of degree $s$ in $\bP^3$, and let $C$ be an effective Cartier divisor on $S$.
Suppose $\sI_C$ is $r$-regular in the sense of Castelnuovo-Mumford.
If $s \geq 2r+1$, then  $C$ is reconstructed  at level $r$ by
 $I_{\alpha(C)}$.
\end{cor}
\begin{proof}
Since $\sI_C$ is $r$-regular, the ideal $I_C$ is generated by its forms of degree $\leq r$ and $H^1 (\sI_C(n))=0$ for every $n \geq r-1$.
As $s \geq 2r+1$ and $r \geq 1$, the first condition $h^1\sI_C(2s-4-r)=0$  in Proposition \ref{condizioni} is satisfied for $p=r$.

 We are left to check that $h^0\cN_{C/\bP^3}(r-s)=0$.

By \cite[Prop 4.1]{PeskineSzpiro}, there are two surfaces $S_1$ and $S_2$ of degree $r$ meeting properly
in a complete intersection $$X=S_1 \cap S_2=C \cup D$$
so that $C$ and $D$ have no common component.
Consider the exact sequence
$$0 \to \sI_{X} \to \sI_C \to \sI_{C,X}\to 0.$$
Applying $\Hom(-,\cO_C)$ we get
$$0 \to \Hom(\sI_{X},\cO_C) \to \cN_C \to {\cN_{X}}_{|C}$$
and $\Hom(\sI_{X},\cO_C)=0$ since $C$ and $D$ have no common component. Therefore, there is an inclusion
$$\cN_C \hookrightarrow (\cN_{X})_{|C}=\cO_C(r) \oplus \cO_C(r)$$
hence $h^0\cN_C(m)=0$ for $m \leq -r-1$. In particular
$h^0\cN_C(r-s)=0$ because $s \geq 2r+1$.
\end{proof}

\section{Arithmetically Cohen-Macaulay curves} \label{4}
In this section we explain how Example 1.15.3 in \cite{EP} extends the result about the perfection of
complete intersections to the much larger class of arithmetically Cohen-Macaulay curves (from now on, ACM curves).
Recall that a curve $C \subset \bP^3$ is called ACM if its homogeneous ring $R_C=R/I_C$ is Cohen-Macaulay,
or, equivalently, if $C$ is locally Cohen-Macaulay of pure dimension $1$ and $H^1_* (\sI_C)=0$. A smooth ACM curve
is what classically was referred to as a {\em projectively normal curve}.  We refer the reader to \cite{HS_ACM} for a detailed study of ACM curves on a surface in $\bP^3$.

If $C \subset \bP^3$ is an ACM curve, then $I_C$ has a free graded resolution of the form

\begin{equation}\label{reslACM}
    0 \to E= \bigoplus_{j=0}^r R(-b_j) \stackrel{\phi}{\longrightarrow} F=\bigoplus_{i=0}^{r+1} R(-a_i)  \to I_C \to 0
\end{equation}
and $I_C$ coincides with the ideal generated by the $r \times r$ minors of $\phi$ by the Hilbert-Burch theorem -
cf. \cite[Proposition II.1.1 p. 37]{MDP}.

Applying the functor $\Hom_R( \bullet, R/I_C)$ to (\ref{reslACM}) as in \cite[p. 428]{Ellingsrud} one obtains  a long exact sequence

\begin{equation}\label{NormalBundleACM}
    0 \to H^0_* (\cN_C)
    \longrightarrow
    \bigoplus_{i=0}^{r+1} R_C(a_i)
    \longrightarrow
    \bigoplus_{j=0}^{r} R_C(b_j)
    \longrightarrow
    H^0_* (\omega_C(4))
     \to 0
\end{equation}

The importance of this sequence for our purposes is that it allows to compute the Hilbert function
$n \mapsto h^0 (\cN_{C,\bP^3}(n))$ of $\cN_{C,\bP^3}$ as a function of the Hilbert function $n \mapsto h^0 (\cO_C(n))$ of $C$;
we can then compute the dimension of $\Ann (\beta(C))_n$ and of $I_{\alpha(C),n}$ in terms solely of the Hilbert
function of $C$ and of the degree $s$ of $S$. To justify our assertion, one needs to observe that to compute
$h^0 (\cN_{C,\bP^3}(n))$ out of (\ref{NormalBundleACM})
one does not need to know the numbers $a_i$'s and $b_j$'s, but only for each $n$ the difference
$$
\#\{i: \; a_i=n \, \} - \#\{j: \; b_j=n \, \}
$$
which depends only on the Hilbert function of $C$.

As an application of this argument, we can give for ACM curves a sharp
bound for the smallest integer $n$ such that
$I_{\alpha(C),n}=I_{C,n}$. For this we will not need
the full Hilbert function of $C$, but just its
index of speciality $ e:=e(C)=\max \{n\ |\ h^1\left(\cO
(n)\right)=h^2 \left(\sI_C (n)\right) \ >0\} $
and the minimum degree $s(C)$ of a surface containing $C$:
$ s(C)=\min \{n\ |\ h^0\left(\sI_C (n)\right) \ >0\} $.
For an ACM curve $C$, the ideal $\sI_C$ is $e+3$-regular because $H^1_* (\sI_C)=0$.
In particular, the ideal $I_C$ is generated in degrees $\leq e+3$, and $s(C) \leq e+3$.

\begin{thm}\label{acmcurves}
Let $S$ be a smooth surface of degree $s$ in $\bP^3$.
Let $C \subset S$ be an ACM curve, let $s(C)$ be the minimum degree of a surface containing $C$ and let
$e(C)$ be the index of speciality of $C$.

If $s\geq 2e(C)+8-s(C)$ then $I_{\alpha(C),(e+3)}=I_{C,(e+3)}$. Therefore $C$ is reconstructed  at level $e+3$ by $I_{\alpha(C)}$.
\end{thm}

\begin{proof}
The statement follows from Proposition
\ref{condizioni} with $p=e+3$ provided we can show that
$h^0\left(\cN_{C/\bP^3}(e+3-s)\right)=0$. For this we use
the exact sequence (\ref{NormalBundleACM}), which shows that
the maximum $n$ for which
$h^0\cN_{C,\bP^3}(n)=0$ is $n= s(C)-e(C)-5$.
\end{proof}

\begin{rmk}
A twisted cubic curve $C$ is ACM with invariants $s(C)=2$ and $e(C)=-1$. Hence
from Theorem \ref{acmcurves} it follows once more that, if $C$ is contained in a smooth
quartic surface $S$, then $C$ is cut out by quadrics in $I_{\alpha(C,S)}$.
\end{rmk}

\begin{rmk}
Theorem \ref{acmcurves} improves for ACM curves  the bound of Corollary
\ref{boundons} because, since $r=e+3$, then $2e+8-s(C)=2r+2-s(C)$.
\end{rmk}


%

In \cite[Sec 2.3]{MovasatiSertoz}, motivated by the case of complete intersections,
formulate the notion of a {\em perfect class}:
\begin{defn}\label{perfection}
Let $S$ be a smooth surface of degree $s$ in $\bP^3$.
A class $\alpha \in H^1(\Omega_S)/\bC \eta_H \subseteq H^{0} (\cO_S(2s-4))^*$
is \emph{perfect at level $m$} if there exist effective divisors $D_1,\ldots,D_q$
in $S$ such that $I_{\alpha(D_i)}=I_\alpha$ for every
$i=1,\ldots, q$ and
$$
I_{\alpha,j} = \sum_{i=1}^q I_{{D_i},j}+J_{S,j} \quad \mbox{for every $j \leq m$.}
$$
We say the class is \emph{perfect} if $I_\alpha =  \displaystyle \sum_{i=1}^q I_{D_i}+J_S $.
We make the convention that the zero class is perfect - geometrically, this amounts to consider the empty set as a (empty) curve, and is consistent with regarding the zero divisor as an effective divisor.
\end{defn}

\begin{exa}
If $C \subset S$ is the complete intersection of two surfaces meeting properly, then $\alpha(C)$ is perfect
(see \cite[Ex 2.11]{MovasatiSertoz}, \cite[Ex 1.15.2]{EP}, \cite[Prop. 2.14]{Dan} ). If one does not agree that the
zero class is perfect, then one needs to add the condition that $C$ is cut out by two surfaces of degrees $<s= \deg (S)$.
\end{exa}

We now wish to generalize the previous example to the class of ACM curves showing that, if $C$ is ACM, then the class
$\alpha(C)$ is perfect. For this we need to recall more facts from \cite{EP}.
 Suppose the ACM curve $C$ is contained in a smooth surface $S$ of degree $s$ and equation $f=0$. Then
the polynomial $f$ can be written in the form
$$
f= \sum_{i=1}^{r+1} g_i h_i
$$
where the $h_i$'s are the images of the generators of the free module $F$ in the resolution (\ref{reslACM}) of $I_C$,
Since the $h_i$'s are the signed $r \times r$ minors of $\phi$, then polynomial
$f$ is the determinant of the morphism $\psi: E \oplus R(-s) \to F$ obtained adding the
column $[g_1, \ldots, g_{r+1}]^T$ to the matrix of $\phi$: in other words, $\psi$
coincides with $\phi$ on $E$, and sends $1 \in R(-s)$ to $\ds \sum_{i=1}^{p+1} g_i e_i$, where the $e_i$'s
are the generators of $F$. We thus obtain a resolution of $I_C/I_S$:
\begin{equation}\label{resIdealCS}
    0 \to E \oplus R(-s) \stackrel{\psi}{\longrightarrow} F \to I_C/I_S \to 0
\end{equation}
Since $S$ is smooth, the curve $C$ is Cartier on $S$ so that $I_C/I_S$ can locally be generated by one element.
It follows that the ideal $I_r (\psi)$ generated by the $ r \times r$ minors of $\psi$ is irrelevant, that is, its
radical is the irrelevant maximal ideal $(x,y,z,w)$ of the polynomial ring $R$.

\begin{prop}\cite[Prop. 1.16]{EP} \label{prop116EP}
Let $C\subset \bP^3$ be an ACM curve contained in the smooth surface $S$. Suppose $I_C$ has the resolution
(\ref{reslACM}). Then
\begin{enumerate}
  \item  if $\psi$ is as in exact sequence (\ref{resIdealCS}) the presentation of $I_{C,S}$, then
  $$I_{\alpha(C)}= I_r (\psi)$$
  is the ideal generated by the $r \times r$ minors of $\psi$;
  \item the $n^{th}$-graded piece $\Ann(\alpha(C))_n$ of the annihilator of $\alpha(C)$ in $\ringS$ is the
  image of the natural map
  \begin{equation*}
    \bigoplus_{m \in \bZ} H^0\cO_S(C+(n+m)H) \otimes H^0 \cO_S(-C-mH) \longrightarrow H^0 (\cO_S(n))
\end{equation*}
\end{enumerate}
\end{prop}

\begin{rmk}
Note that $\Ann(\alpha(C))= I_{\alpha(C)}/I_S$. The  equality
$I_{\alpha(C)}= I_r (\psi)$ is a non trivial fact that is not given a full proof in \cite{EP}; a complete proof can be found in \cite[Proposition 4.3 and p. 382]{KM}.
\end{rmk}

We can now prove that the class $\alpha(C)$ of an ACM curve in a smooth surface  $S$ is perfect:

\begin{thm}
\label{acmperfette}
Let $C \subset S$ be an ACM curve and let $S$ be a smooth surface. Then the class $\alpha(C)$ of $C$ in $S$ is perfect.
\end{thm}
\begin{proof}
Fix an integer $n$.
By Proposition \ref{prop116EP} $\Ann(\alpha(C))_n$ is the image of the natural map
  \begin{equation*}
    \bigoplus_{m \in \bZ} H^0\cO_S(C+(n+m)H) \otimes H^0 \cO_S(-C-mH) \longrightarrow H^0 (\cO_S(n)).
\end{equation*}
Note the sum on the left hand side is finite, and consists of those $m$ for which
the linear systems $|C+(m+n)H|$ and $|-C-mH|$ are both non-empty. For such an $m$
we pick a basis $g_1, \ldots, g_{r_m}$ of $ H^0 \cO_S(-C-mH)$ and corresponding
effective divisors $D_k=(g_k)_0 \in |-C-mH|$. The image of
$$H^0\cO_S(C+(n+m_k)H) \otimes  g_k $$ in $H^0 (\cO_S(n))$ is $H^0\sI_{D_k/S}(n)$.
(if $C \sim tH$ is a complete
intersection of $S$ and another surface, taking $m_k=-t$ and $n=0$ we get
$D_k$ the empty curve, and in this case $\alpha(C)=0$ is perfect by our
definition).
Note that $\bQ \alpha(D_k)=\bQ \alpha(C)$ by Proposition \ref{vanishing}.
Now letting $k$ and $m$ vary we see
 that $\alpha(C)$ is perfect at level $n$, for every $n$. Since $\Ann(\alpha(C))$ is finitely generated,
we can let $n$ vary up to the maximum degree of a generator of $\Ann(\alpha(C))$, and
 recover  the whole $\Ann(\alpha(C))$ as the sum of finitely many $I_{D_{k}/S}$ with
 $D_{k} \sim  C+(n+m)H$ for some $m$ and $n$. Therefore $\alpha_C$ in $S$ is perfect.

%
\end{proof}

\section{Example of a non perfect class}\label{sec:rational}

\bt\label{rationalquintic}
Let $C \subset \bP^3$ be a smooth rational curve of degree $4$ contained in a
smooth surface $S$ of degree $s=4$. The class $\alpha(C)$ in $S$ is not
perfect at level $3$.

\et
\begin{proof}
A smooth rational quartic curve $C \subset \bP^3$ is contained in a unique quadric surface $Q$,
and $Q$ is necessarily smooth (all curves on the quadric cone are arithmetically Cohen-Macaulay by
\cite[Chapter V Ex. 2.9 ]{Hartshorne}). We may assume
$C$ is a  divisor of type $(3,1)$ on $Q$. The ideal sheaf of $C$ is $3$-regular, hence
$I_C$ generated by quadrics and cubics.

Suppose $C$ is contained in a smooth quartic surface $S$. Then $Q \cap S$ is the union of
$C$ and an effective divisor $D_0$ of type $(1,3)$ on $Q$. Note that $D_0$ is a curve
of degree $4$ and arithmetic genus $0$; as the divisor class of $D_0$
is different from that of $C$ and $C$ is irreducible, we conclude that $C$ and $D_0$ have no common component.

The curves $C$ and $D_0$ don't move in their linear system on the quartic surface $S$: for $C$ this follows
from $C^2=-2$, and in any case for both $D_0$ and $C_0$ one might argue that
$$
h^0 (\cO_S(D_0))= h^0 (\cO_S(2H-C))= h^0 (\sI_C(2))=1.
$$

Having established the geometric set-up, we proceed to show that  $I_{\alpha(C)}$ contains
too many cubics for $\alpha(C)$ to be perfect at level $3$.
To compute the dimension of $I_{\alpha(C),3}$, we use the fact that $R/I_{\alpha(C)}$ is a Gorenstein ring with socle in degree $2s-4=4$, hence
$$
\dim I_{\alpha(C),3}=\dim  I_{\alpha(C),1}+\dim R_3-\dim R_1= \dim  I_{\alpha(C),1}+16 \geq 16.
$$
This estimate is good enough for us to prove the theorem, but let us show anyway that
$\dim I_{\alpha(C),3}=16$:  as $C$ is a divisor of type $(3,1)$ on $Q$,
$h^1\left( \sI_C(3)\right)=0$  hence by Lemma \ref{propannihilators} $I_{\alpha(C),1}$ is the pull back to $R_1$ of
$N(C)_1$, the image of $H^0\hspace{-3pt}\left(\cN_{C/\bP^3} (-3)\right)$ in $H^0 \left( \cO_{C}(1)\right)$; as the normal bundle of $C$
pulls-back on $\bP^1$ to $ \cO_{\bP^1}(7) \oplus
\cO_{\bP^1}(7)$ by \cite[Proposition 6]{Eisenbud-VdVen1}), we conclude
that $I_{\alpha(C),1}=0$, hence
$
\dim I_{\alpha(C),3} =16.
$ The same argument shows that $I_{\alpha(C),2}=I_{C,2}$ as well.

%

To check whether $I_{\alpha,3}$ is perfect, we need to determine curves $D$ in $S$ with
$I_{\alpha(D)}=  I_{\alpha(C)}$ and  $h^0 \left( \sI_D(3)\right) \geq 1$ so that
$D$ can contribute to $I_{\alpha,3}$.
Thus suppose $D$ is  such a curve. By
Proposition \ref{vanishing}, there exist $m,n,p \in \bZ$, $m,n \neq 0$ and
relatively prime, such that $pH+mC+nD$ is linearly equivalent to zero.
By assumption $3H-D$ is effective; as $C$ is not linearly equivalent to $tH$ for any $t$,
neither is $D$, hence $1 \leq \deg(D)=D \cdot H \leq 11$. Replacing
$D$ with $D'=3H-D$ we can even assume $D\cdot H \leq 6$.

Now consider the matrix
\[
M=
\begin{bmatrix}
H^2 & C \cdot H & H\cdot D \\
C\cdot H & C^2 & C\cdot D \\
H\cdot D & C \cdot D & D^2
\end{bmatrix}=
\begin{bmatrix}
4 & 4 & x \\
4 & -2 & y \\
x & y & z
\end{bmatrix}
\]
As $pH+mC+nD$ is linearly equivalent to zero, the vector $v=[p,m,n]^T$
is in the kernel of $M$. Set $x=H\cdot D$, $y=C\cdot D$ and $z=D^2$. Note that
 $z=D^2=2(p_a(D)-1)=2q$ is even.

The determinant of $M$ must vanish, so
\[x^2+4xy-2y^2-24q=0
\]
From this we deduce first that $x$ and $y$ must be even, and then that $4$ divides $x$.
As $1 \leq x \leq 6$, we must have $x=4$. Thus $D$ is a curve of degree $4$, and either $D=C$ or
$C$ is not a component of $D$, hence
$y=C \cdot D \geq 0$. Assume that $D \neq C$.
Writing $y=2t$ with $t\geq0$, we obtain the equation
$$
t^2-4t+3q-2=0
$$
Looking at the discriminant of this quadratic equation in $t$ we deduce $6-3q$ is a perfect square,
so that $q=2-3a^2$ for an integer $a \geq 0$. Then solving for $t$ and imposing $t \geq 0$
we obtain $t=2+3a$. So $H\cdot D=x=4$, $C \cdot D=y=4+6a$ and $D^2=4-6a^2$.
Then solving the linear system $Mv=0$ for $v=[p,m,n]^T$ we find $m=an$ and $p=-(a+1)n$.
Since $m$ and $n$ are relatively prime and non zero and $a \geq 0$, the only possibility is that $a=1$. Then we can take $m=n=1$ and conclude $C+D \sim 2H$,
so that $C+D$ is the complete intersection of the unique quadric $Q$
containing $C$ with $S$, and  $D=D_0$ is the residual to $C$ in the
complete intersection $Q \cap S$.

We conclude that the only curves $D$ in $S$ that are contained in a cubic surface
and satisfy $I_{\alpha(D)}=  I_{\alpha(C)}$ are
$C$, the residual $D_0$ to $C$ in the
complete intersection $Q \cap S$, and the effective divisors linearly equivalent to either $3H-C$ or $3H-D_0$.
But observe that, if $D' \sim 3H-D_0 \sim C+H$ is effective, then
$$
h^0 \sI_{D'}(3)= h^0 \sI_{C}(2)=1.
$$
Therefore there is a unique cubic containing $D'$, whose equation is contained in the ideal of $D_0$. Similarly,
if $D^{''} \sim 3H-C$ is effective, there is a unique cubic containing $D^{''}$,
whose equation is contained in the ideal of $C$.
Hence any cubic form that belongs to the ideal of a curve $D$ on $S$ satisfying $I_{\alpha(D)}=  I_{\alpha(C)}$
is in the vector space spanned by $I_{C,3}$ and $I_{D_0,3}$.

To show $\alpha(C)$ is not perfect at level $3$ it is now enough to show
that cubics containing either $C$ or $D_0$ plus the cubics in the Jacobian ideal
$J_S$ do not span $I_{\alpha(C),3}$.

To this end, note that cubic surfaces that contain both $C$ and $D_0$ are in
the ideal of the complete intersection of $S$ and $Q$, and so form a vector
space of dimension $4$. By Grassmann's formula
$$
\dim I_{C,3} + \dim I_{D_0,3} = 7+7-4=10
$$
There are four independent cubics in the Jacobian ideal, so
$$ \dim I_{C,3} + \dim I_{D_0,3} + \dim J_{S,3} \leq 14 <16 =
\dim I_{\alpha(C),3}
$$
and this shows that $\alpha(C)$  in $S$ is not perfect at level 3.
\end{proof}

\bibliographystyle{alpha}
\bibliography{ebib}

\end{document}